\documentclass[12pt]{amsart}
 
\usepackage{amsmath}
\usepackage{bbm}
\usepackage[utf8]{inputenc}
\usepackage[OT1]{fontenc}
\usepackage{txfonts}
\usepackage{color}
\usepackage{epic}
\usepackage{eepic}
\usepackage{graphicx}
\usepackage{paralist}
\usepackage[tight,hang,TABTOPCAP]{subfigure}
\usepackage[small,it]{caption}
\usepackage{tabularx,booktabs}
\usepackage[mathscr]{eucal}
\usepackage{units}
\usepackage{tikz}
\usepackage{upgreek} 
\usepackage{multirow}

\usepackage{url}
\usepackage{color}
\definecolor{darkblue}{rgb}{.8,.15,.15}
\definecolor{darkgreen}{rgb}{0.15,.4,.5}
\usepackage[pagebackref,
            naturalnames,
            colorlinks,linkcolor=darkblue,
            citecolor=darkblue,
            filecolor=darkblue,urlcolor=darkgreen,
            anchorcolor=darkblue,menucolor=darkblue,
            bookmarksopen=false,bookmarks=false
           ]{hyperref}

\newcommand{\tetra}{\Delta'}
\renewcommand*{\backref}[1]{}
\renewcommand*{\backrefalt}[4]{%
\ifcase #1 %
 [Not cited]%
\or
 [Cited on page #2]%
\else
 [Cited on pages #2]%
\fi
}

\graphicspath{{./figures-ziegler/}{./figures/}}

\newtheorem{theorem}{Theorem}[section]
\newtheorem{corollary}[theorem]{Corollary}
\newtheorem{lemma}[theorem]{Lemma}
\newtheorem{proposition}[theorem]{Proposition}

\theoremstyle{definition}
\newtheorem{definition}[theorem]{Definition}
\newtheorem{example}[theorem]{Example}
\newtheorem{remark}[theorem]{Remark}

\numberwithin{theorem}{section}
\numberwithin{equation}{section}
\numberwithin{table}{section}

\newcommand{\R}{\mathbb{R}}
\newcommand{\Z}{\mathbb{Z}}

\newcommand{\N}{\mathbb{N}}

\newcommand{\conv}{\operatorname{conv}}

\title[Unimodular triangulations of dilated 3-polytopes]{
  Unimodular triangulations of dilated 3-polytopes}

\author[Santos]{Francisco Santos}
\address[Francisco Santos]{Facultad de Ciencias, Universidad de Cantabria, 
Av. de los Castros s/n, E-39005 Santander, Spain.}  
\email{francisco.santos@unican.es}
\thanks{Work of F. Santos is supported in part by 
 the Spanish Ministry of Science under grants MTM2011-22792 and 
by MICINN-ESF EUROCORES programme EuroGIGA -- ComPoSe -- IP04 (Project EUI-EURC-2011-4306). Part of this work was done while I was visiting FU Berlin in 2012 and 2013 supported by a Research Fellowship of the Alexander von Humboldt Foundation}
\author[Ziegler]{G\"unter M.~Ziegler}
\address[G\"unter M. Ziegler]{Inst.\ Mathematics, FU Berlin, Arnimallee 2, 14195 Berlin, Germany.} 
\email{ziegler@math.fu-berlin.de}
\thanks{Work of G. M. Ziegler is supported by the European Research
  Council under the European Union's Seventh Framework Programme (FP7/2007-2013)/\allowbreak ERC
  Grant agreement no.~247029-SDModels and by the DFG Research Center \textsc{Matheon} ``Mathematics
  for Key Technologies'' in Berlin.}

\date{April 26, 2013}
\begin{document}
\begin{abstract}
  A seminal result in the theory of toric varieties, due to Knudsen, Mumford and Waterman (1973), asserts that for every lattice polytope $P$ there is a positive integer $k$ such that the dilated polytope $kP$ has a unimodular triangulation. In dimension 3, Kantor and Sarkaria (2003) have shown that $k=4$ works for every polytope. But this does not imply that every $k>4$ works as well. We here study the values of $k$ for which the result holds showing that:
\begin{enumerate}
\item It contains all composite numbers.
\item It is an additive semigroup.
\end{enumerate}

These two properties imply that the only values of $k$ that may not work (besides $1$ and $2$, which are known not to work) are $k\in\{3,5,7,11\}$. With an ad-hoc construction we show that $k=7$ and $k=11$ also work, except in this case the triangulation cannot be guaranteed to be ``standard'' in the boundary. All in all,  the only open cases are $k=3$ and $k=5$.
\vspace{-3\baselineskip}
\end{abstract}
\maketitle

{\tiny
\tableofcontents}

\section{Introduction}

Let $\Lambda\subset \R^d$ be an affine lattice. A \emph{lattice polytope} (or an \emph{integral polytope}~\cite{KKMS3}) is a polytope $P$ with all its vertices in $\Lambda$. We are specially interested in lattice (full-dimensional) simplices. The vertices of a lattice simplex $\Delta$ are an affine basis for $\R^d$, hence they induce a sublattice $\Lambda_\Delta$ of $\Lambda$ of index equal to the \emph{normalized volume} of $\Delta$ with respect to $\Lambda$. A lattice simplex is \emph{unimodular} if its normalized volume is~$1$, that is, if its vertices are a lattice basis for $\Lambda$.

In many contexts it is interesting to triangulate a given lattice polytope $P$ into unimodular simplices, that is, to construct a \emph{unimodular triangulation} of $P$. This is not possible for all lattice simplices of dimension $d\ge3$, but there is the following important result of Knudsen, Mumford and Waterman~\cite{KKMS3} (a proof appears also in \cite{bgBook}):

\begin{theorem}[\protect{\cite{KKMS3}}]
\label{thm:KMW}
For every lattice polytope $P$ there is a constant $k$ such that the dilation $kP$ admits a unimodular triangulation.
\end{theorem}

The constant $k$ needed in the theorem can be arbitrarily large, as the following example shows:

\begin{example}
\label{exm:largek}
Let $\Delta=\conv\{\mathbf{0},e_1,\dots,e_d\}$ be the standard $d$-simplex in $\Z^d$, but consider it with respect to the lattice
$\Lambda=\Z^d\cup\big(\big(\frac{1}{2},\dots,\frac{1}{2}\big)+\Z^d\big)$, so that $\Delta$ has volume~$2$. For any $k<d/2$ we have $k\Delta\cap \Lambda =k\Delta\cap \Z^d$, which implies that $k\Delta$
contains no unimodular simplex with respect to $\Lambda$. In particular, it does not have a unimodular triangulation.
\end{example}

One question that is open, though, is whether a constant $k=k_d$ can be chosen that depends only on the dimension of~$P$. The above example provides a lower bound $k_d\ge\lceil d/2\rceil$ for this constant, in case it exists.
In this respect, Kantor and Sarkaria have shown the following.

\begin{theorem}[\protect{\cite{KantorSarkaria}}]
\label{thm:KS}
For every $3$-dimensional lattice polytope $P$ the dilation $4P$ admits a unimodular triangulation.
\end{theorem}

The fact that $kP$ has a unimodular triangulation does not automatically imply that $k'P$ has one for every $k'>k$. The only general result in this direction is that this must hold whenever $k'$ is a multiple of $k$. In this paper we show the following.

\begin{theorem}
\label{thm:kP}
For every $3$-dimensional lattice polytope $P$ and every $k\ge 6$, $kP$ has a unimodular triangulation.
\end{theorem}

This result, as well as the one from~\cite{KantorSarkaria}, is based on understanding \emph{empty} lattice simplices in dimension 3. A lattice simplex is called \emph{empty} if the only lattice points it contains are its vertices.\footnote{Kantor and Sarkaria~\cite{KantorSarkaria} use ``primitive'' and ``elementary'' for unimodular and empty, respectively. Scarf~\cite{Scarf} uses ``integral'' for empty.} 
Every unimodular simplex is empty, but the converse fails for empty simplices in all dimensions $d\ge3$, as exemplified by
$\conv\{e_1,\dots,e_d,\mathbf{1}\}$
 (which for $d=3$ yields an empty regular tetrahedron of normalized volume $2$ inscribed in the unit cube).%

A relation between empty and unimodular simplices comes from the following natural idea on how to triangulate $kP$. Every lattice $3$-polytope $P$ can trivially be triangulated into empty tetrahedra, so $kP$ can be triangulated into $k$-th dilations of empty tetrahedra. If, for a certain $k$, we have that every $k$-th dilation of an empty tetrahedron has a unimodular triangulation, then $kP$ has a decomposition into unimodular simplices that do not overlap one another. If, moreover, the triangulations of the individual dilated tetrahedra are made to agree in the facets they have in common, then we have a unimodular triangulation of $kP$. The simplest way to achieve this agreement of the individual triangulations is to make them \emph{standard on the boundary} in the following sense:

\begin{definition}
\label{defi:standard}
Let $\Delta\subset \R^3$ be an empty lattice tetrahedron and let $T$ be a lattice triangulation of $k\Delta$, for some $k$. We say that $T$ is \emph{standard on the boundary} if all the boundary edges in $T$ are parallel to edges of $\Delta$.
\end{definition}

This property uniquely defines the triangulation $T|_{\delta(\Delta)}$: Each facet of $T$ is an empty (hence unimodular) triangle $F$ in a $2$-dimensional lattice, so that $kF$ has a unique unimodular triangulation using only edges parallel to those of $F$ (see Figure~\ref{fig:standard}).
\begin{figure}
\includegraphics[scale=.5]{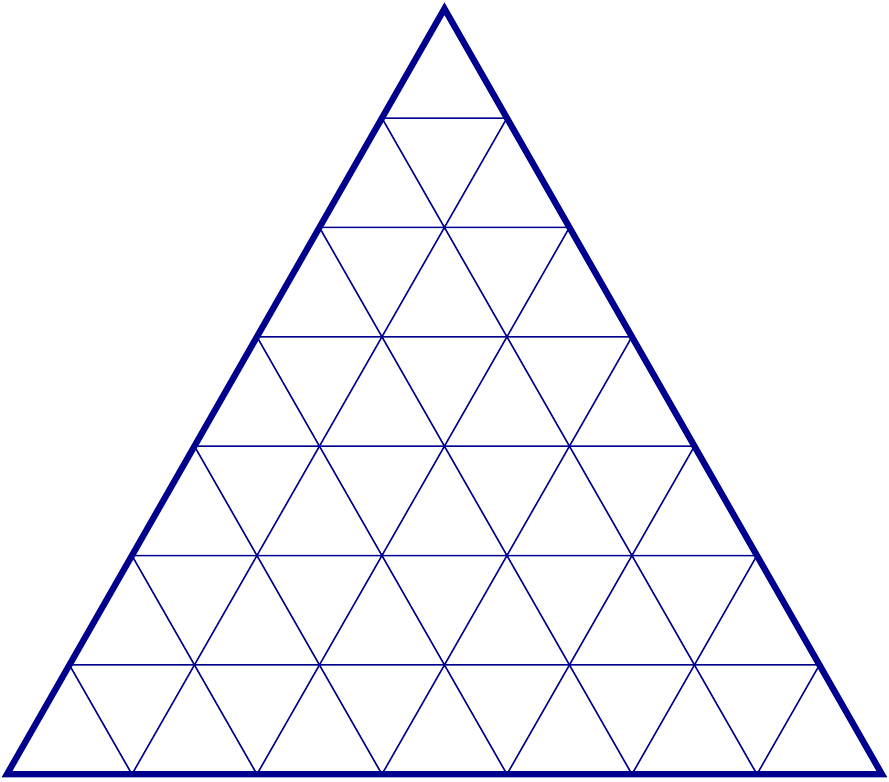}
\caption{Standard triangulation of a dilated triangle, for $k=7$.}
\label{fig:standard}
\end{figure}

With this in mind we can now be more precise about the statements in this paper:
\begin{compactitem}[ $\bullet$ ]
\item In Section~\ref{sec:k=2} we establish a characterization of the empty tetrahedra $\Delta$ for which $2\Delta$ admits a unimodular triangulation.  
For all of these the unimodular triangulation can be chosen to have standard boundary.
This result appears also in~\cite{KantorSarkaria}, but we repeat the proof since the techniques that appear in it will be useful later.

\item In Section~\ref{sec:non-standard} we show that $k\Delta$ has a unimodular triangulation for every empty tetrahedron $\Delta$ and every $k\ge4$. However, the triangulation we construct is not, in general, standard in the boundary.

\item In Section~\ref{sec:standard} we extend Kantor and Sarkaria's technique to apply to $k\Delta$ for every composite $k$; that is, for every empty tetrahedron $\Delta$ and every composite positive integer $k$ we construct a unimodular triangulation of $k\Delta$ with standard boundary. We also show that whenever $k_1\Delta$ and~$k_2\Delta$ have unimodular triangulations with standard boundary then so does $(k_1+k_2)\Delta$. This implies Theorem~\ref{thm:kP}, except for the cases $k=7$ and $k=11$.

\item In Section~\ref{sec:semi-standard} we define the concept of triangulation with \emph{semi-standard boundary} and show that $k\Delta$ has a unimodular triangulation with semi-standard boundary for every $k\ge 7$ and every empty tetrahedron $\Delta$. 
This finishes the proof of Theorem~\ref{thm:kP}.
\end{compactitem}

Some comments on the history of this paper are in order. The results in Sections~\ref{sec:k=2} and~\ref{sec:non-standard} were 
found by the second author in 1997 (based on discussions with Jeff Lagarias) 
during the special year on \emph{Combinatorics} at the Mathematical Sciences Research 
Institute (Berkeley, California). They were presented in several lectures (in Berkeley, New York, Oberwolfach, and Barcelona) that year.
Our Figures~\ref{fig:square_5_13},~\ref{fig:max-paths},~\ref{fig:half-2Delta},~\ref{fig:layers},~\ref{fig:toblerone}, and~\ref{fig:alternating} are reproduced from the overhead
slides that were used for those talks. The results of Sections~\ref{sec:standard} and~\ref{sec:semi-standard} were found by the first author in 2009 
during the workshops  \emph{Combinatorial Challenges in Toric Varieties} (American Institute of Mathematics, Palo Alto, California) 
and \emph{Combinatorial Geometry} (Institute for Pure and Applied Math, Los Angeles, California). We thank the three institutions 
and the organizers of these events for their support and the excellent scientific environments they provided.

\section{Classification of $3$-dimensional empty simplices}
\label{sec:prelim}

For any polytope $P\in \R^d$ and a linear functional $f:\R^d\to \R$, the \emph{width} of~$P$ with respect to $f$ is the difference between the maximum and minimum values taken by $f$ on $P$. If $P$ is a lattice polytope we are interested in $f$ 
belonging to the dual lattice, in which case the width of $P$ with respect to $f$ is an integer.
Thus, if we speak of width in the following, we always refer to \emph{lattice width}, which is the minimal width of the
lattice polytope with respect to any non-constant dual lattice functional.

All the results in this paper use the classification of $3$-dimensional empty simplices by White \cite{White}. But, more than that, they use in an essential way the property that every empty lattice $3$-simplex has width one (from which the classification is easy to derive, as we show below).

For any full-dimensional lattice polytope $P$ the minimum possible width is~$1$, which occurs if and only if all the vertices of $P$ lie in two consecutive $f$-constant lattice hyperplanes. In fact, we can also speak about the \emph{width with respect to a lattice hyperplane $H$}, by which we mean width with respect to the unique (modulo sign) primitive linear functional that is constant on~$H$. Similarly, the following theorem speaks of width of a tetrahedron with respect to two opposite edges, meaning ``with respect to the unique primitive functional that is constant on both edges.''
The theorem does not extend to higher dimensions.  
For example, Seb\H{o}~\cite{Sebo} has shown that in all dimensions $d\ge4$ there are empty simplices that have width at least $d-2$ with respect to any lattice hyperplane.

\begin{theorem}
\label{thm:width-one} 
Every empty lattice tetrahedron has width one with respect to (at least) one of its three pairs of opposite edges.
\end{theorem}

This has the following interpretation, which is implicitly used in what follows: 
Consider $\R^3$ tiled by translated copies of a  
tetrahedron and an octahedron, in the natural way (this is, for example, the Delaunay tiling of the face centered cubic lattice). Let $\Lambda$ be a lattice that contains the vertices of this tiling. Theorem~\ref{thm:width-one} says that if the tetrahedra are empty with respect to $\Lambda$ then the lattice points in each octahedron $P$ lie all (except of course for two of the vertices) in one of the three quadrilaterals formed with the vertices of $P$.

Theorem~\ref{thm:width-one} and the corollary below were first proved by White in 1964 \cite{White}, although Scarf~\cite{Scarf} attributes it to Howe (unpublished). A proof, and an extension, can also be found in~\cite{Reznick}. The tetrahedra that appear in the classification had been previously constructed by Reeve~\cite{Reeve}.

\begin{corollary}
\label{coro:classification}
Let $\Delta$ be an empty lattice tetrahedron in $\R^3$ of (normalized) volume $q>1$. Then, $\Delta$ is unimodularly equivalent to the following tetrahedron, for some $p\in\{1,2,\dots,q-1\}$:
\[
\Delta(p,q):=\conv\{(0,0,0), (1,0,0), (0,0,1), (p,q,1)\}
\]
Moreover,
\begin{enumerate}[ \rm(1) ]
\item $\gcd(p,q)=1$, and
\item $\Delta(p,q)\cong \Delta(p',q)$ if and only if $p'\equiv\pm p^{\pm1} \pmod q$.
\end{enumerate}
\end{corollary}

\begin{proof}
Let $a$, $b$, $c$ and $d$ be the vertices of $\Delta$. Our assumption is that the edges $ab$ and $cd$ lie in consecutive lattice hyperplanes, and there is no loss of generality in assuming that these are the hyperplanes $z=0$ and $z=1$. Also, since the edge $ab$ is primitive, we can assume $a=(0,0,0)$ and $b=(1,0,0)$. Now, every integer translation of the plane $z=1$ extends to a unimodular isomorphism that restricts to the identity on $z=0$. So, we can also assume $c=(0,0,1)$ and call $p$ and $q$ the first and second coordinates of $d$. By a reflection on the $y=0$ plane we assume $q>0$. Finally, if $p\not\in[0,q]$ we consider the transformation $(x,y,z)\mapsto(x-\left \lfloor
\frac{p}{q}\right\rfloor y,y,z)$ to the whole tetrahedron, which fixes $a$, $b$ and $c$.

If $\gcd(p,q)\ne 1$, then the edge $cd$ is not primitive, hence $\Delta$ is not empty. That $p'\equiv\pm p^{\pm1} \pmod q$ is necessary and sufficient for $\Delta(p,q)\cong \Delta(p',q)$ is easy to show.
\end{proof}

If we extend the family of canonical tetrahedra to include $\Delta(0,1)$ (which is unimodular), the corollary gives a full set of representatives for the empty lattices $3$-simplices modulo unimodular equivalence. The representatives are not unique, however, as documented by part~(2) of the corollary.

It is now easy to describe the set of lattice points in a dilated standard simplex $k \Delta(p,q)$. For what follows, it is convenient to do so with respect to a more intrinsic system of coordinates.

\begin{lemma}
\label{lemma:lambda_pq}
Let $\ell_{p,q}:\R^3\to \R^3$ be the linear map 
\[
(x,y,z)\mapsto \left(qx -py , y ,z\right).
\]
Then
\begin{enumerate}[ \rm(1) ]
\item $\ell_{p,q}(\Delta(p,q))$ is the \emph{$q$-th right-angled tetrahedron}  
\[
\conv\{(0,0,0), (q,0,0), (0,0,1),(0,q,1)\};
\]
\item $\ell_{p,q}(\Z^3)$ is the lattice 
\[
\Lambda_{p,q}:=q \Z\times q\Z \times \Z + \Z \left\langle \left({p'}, 1,0\right) \right\rangle = q \Z\times q\Z \times \Z + \Z \left\langle \left(1, {p''},0 \right) \right\rangle,
\] 
where $p' = -p \bmod q$ and $p'' = -p^{-1} \bmod q$.
\end{enumerate}
\end{lemma}

\begin{proof}
For part (1) just evaluate $\ell_{p,q}$ at the four vertices of $\Delta(p,q)$. For part~(2) first observe that $\ell_{p,q}$ has determinant $q$ (its matrix is triangular with product $q$ on the diagonal) so that $\ell_{p,q}(\Z^3)$ is a lattice of determinant $q$. By part (1) this lattice contains $q \Z\times q\Z \times \Z$. If we show it contains both $\left(p', 1,0\right)$ and $\left(1,{p''},0\right)$ we are finished. This is so because
\[ 
\ell_{p,q}(1,1,0)=\left(p', 1,0\right)
\quad \text{ and}\quad
\ell_{p,q}\Big(\tfrac{1+pp''}{q},p'',0\Big)=(1,p'',0).\vspace{-7mm}
\]
\end{proof}

So, from now we denote by $\tetra(p,q)$ the tetrahedron  
\[
\tetra(p,q):=\conv\{(0,0,0), (q,0,0), (0,0,1),(0,q,1)\},
\]
which we consider with respect to the lattice 
\[
\Lambda_{p,q}:=q \Z\times q\Z \times \Z + \Z \left\langle \left({p'}, 1,0\right) \right\rangle = q \Z\times q\Z \times \Z  + \Z \left\langle \left(1, {p''},0 \right) \right\rangle,
\] 
where $p' = -p \bmod q$ and $p'' = -p^{-1} \bmod q$. The remarks above show that $\tetra(p,q)$ is (with respect to $\Lambda_{p,q}$) lattice equivalent to $\Delta(p,q)$ with respect to~$\Z^3$.

\section{Unimodular triangulations of $2\tetra(p,q)$ exist only for $p=\pm1$}
\label{sec:k=2}

Let $\tetra(p,q)$ be an empty tetrahedron. 
Then $2\tetra(p,q)$ contains exactly the following points of the lattice $\Lambda_{p,q}$:
\begin{compactitem}[ $\bullet$ ]
\item Its four vertices, $(0,0,0)$, $(2q,0,0)$, $(0,0,2)$, and $(0,2q,2)$ are lattice points.
\item The six mid-points of edges are lattice points. Four of them are at height~$1$ (here and in what follows we call height the third coordinate); these four, $(0,0,1)$, $(q,0,1)$, $(0,q,1)$ and $(q,q,1)$, are the vertices of a \emph{fundamental square} of the lattice $q\Z\times q\Z \times \{1\}$.
\item There are exactly $q-1$ lattice points in the interior of the fundamental square, one on each line $\{x =i\}$, $i=1,\dots, q-1$, and also one on each line $\{y=j\}$, $j=1,\dots q-1$. More precisely, we have
\[
(ip'\bmod q,i,1)\in \tetra(p,q)\qquad\text{for all }i,\ 1\le i\le q,
\]
and
\[
(j,jp''\bmod q,1)\in  \tetra(p,q)\qquad\text{for all }j,\ 1,\le j\le q, 
\]
where $p'=-p\bmod q$ and $p''=-p^{-1}\bmod q$.
\end{compactitem}
See Figure~\ref{fig:square_5_13} for a picture of the $16$ lattice points in the fundamental square of $\tetra(5,13)$.
Note that in general the lattice points in the interior of the fundamental square form the pattern of a $(q-1)\times(q-1)$ permutation matrix.

\begin{figure}
\includegraphics[scale=.5]{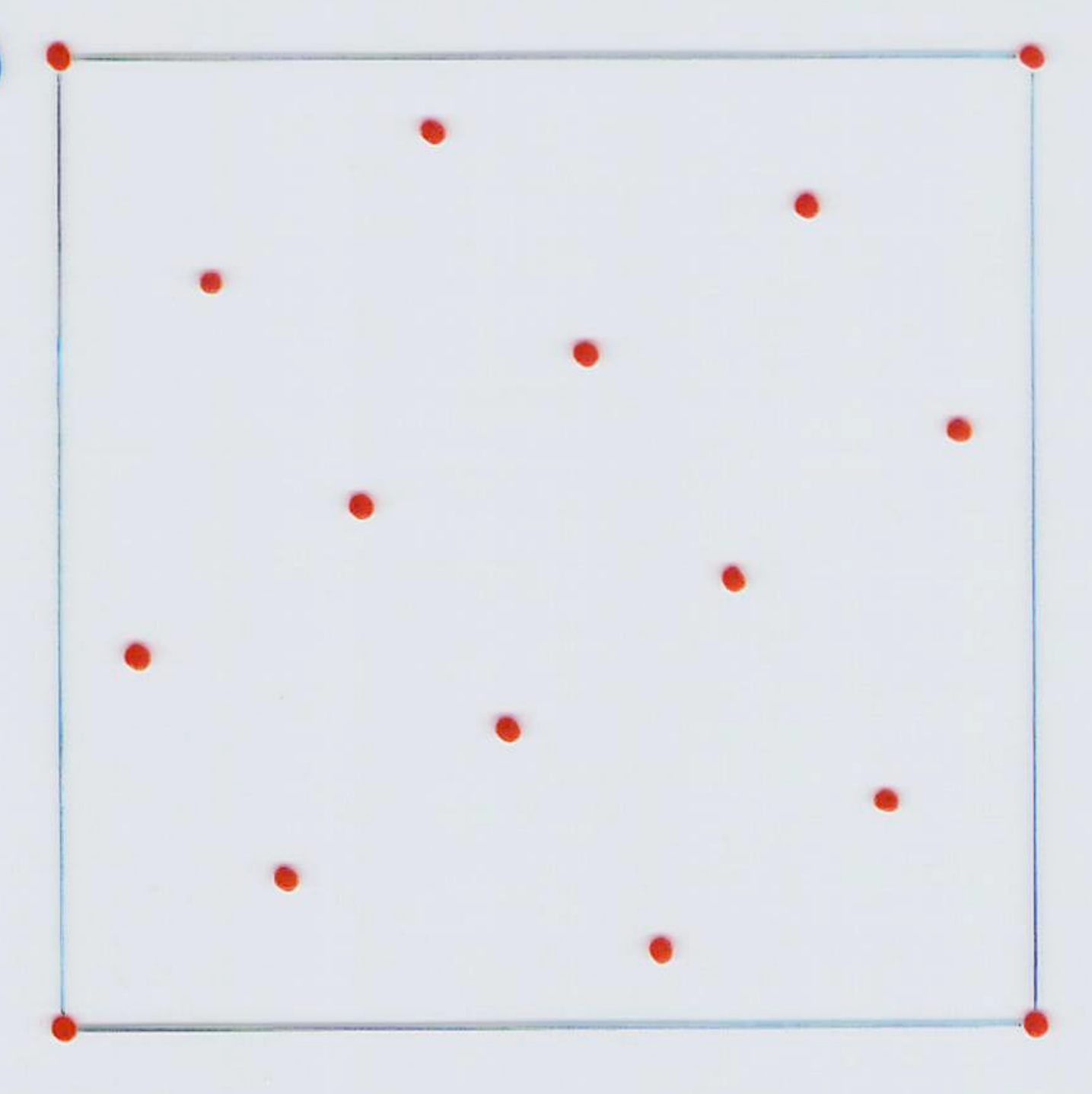}
\caption{The lattice points in the fundamental square for $\tetra(5,13)$.}
\label{fig:square_5_13}
\end{figure}

\begin{definition}
\label{defi:max-path}
We call \emph{maximal X-path} the X-monotone path that goes through the $q-1$ lattice points in the interior of the fundamental square. Similarly, we call \emph{maximal Y-path} the Y-monotone path that goes through the $q-1$ lattice points in the interior of the fundamental square.
See Figure~\ref{fig:max-paths}.
\end{definition}

\begin{figure}
\includegraphics[scale=.5]{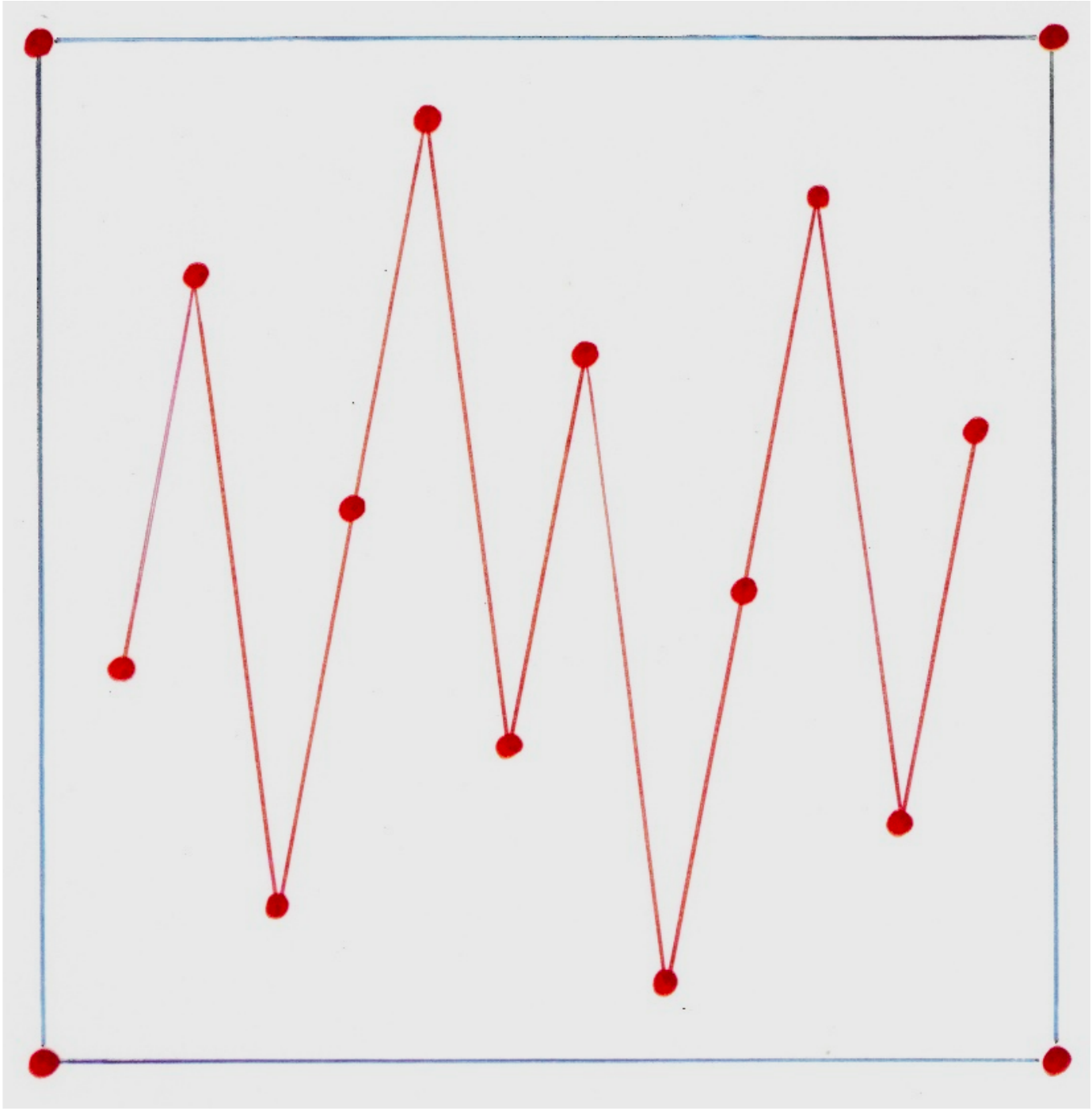}
\caption{The maximal X-path for $\tetra(5,13)$. In this case the maximal Y-path can be obtained by 90 degree rotation, but that is a coincidence, due to the fact that $-5=5^{-1} \bmod 13$.}
\label{fig:max-paths}
\end{figure}

\begin{lemma}\label{lemma:k=2-necessity} 
\begin{compactenum}[\rm(1)]
\item Suppose that the dilated simplex $2\tetra(p,q)$ has a unimodular triangulation.  Then $2\tetra(p,q)$ has a unimodular triangulation $T$ that does not contain any of the edges from height $0$ to height $2$; that is, $T$ triangulates both the half-tetrahedra $\tetra(p,q)\cap\{z\ge1\}$ and $\tetra(p,q)\cap\{z\le 1\}$. 
\item Any $T$ that satisfies the conditions in {\rm(1)} contains all the edges of the maximal X-path and of the maximal Y-path in the fundamental square.
\end{compactenum}
\end{lemma}

\begin{proof}
Let us first prove part (2). For this, we look at the link of one of the bottom edges of $T$, say that of the edge from $(0,0,0)$ to $(q,0,0)$. This link must form a path in the fundamental square, going from one of the vertices $(0,0,1)$ or $(q,0,1)$ to one of $(0,q,1)$ or $(q,q,1)$. For the triangulation to be unimodular, two consecutive vertices in this path must have their second coordinate differing by exactly one. Hence, all the interior lattice points appear, and they do so in their Y order so that the maximal Y-path appears in full the link. The same argument, for one of the top edges, shows that the maximal X-path appears also. Figure~\ref{fig:half-2Delta} illustrates this.
\begin{figure}
\includegraphics[scale=.5]{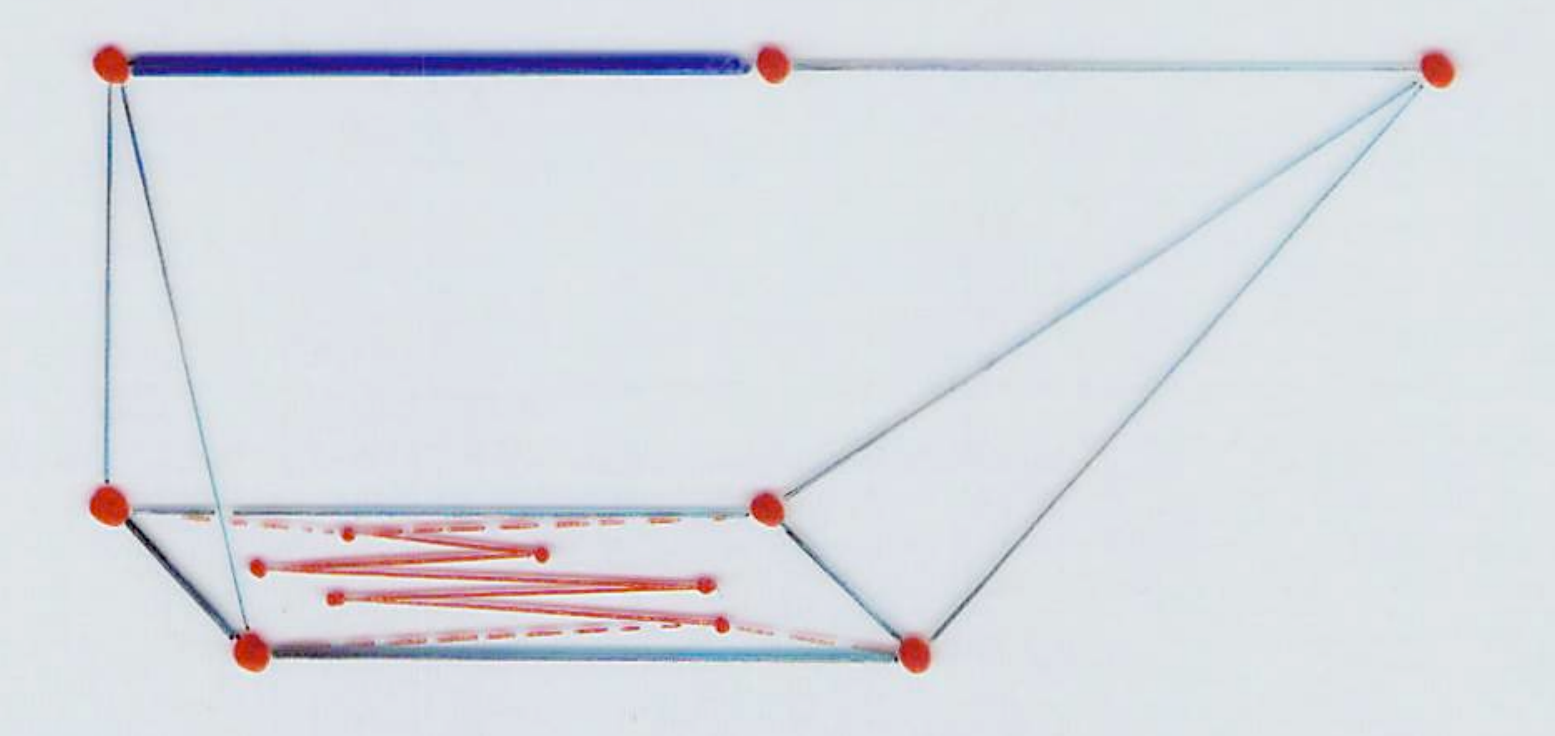}
\caption{Triangulating the upper half of the dilated tetrahedron $2\tetra(p,q)$.}
\label{fig:half-2Delta}
\end{figure}

Let us now prove part (1). For this, let $T$ be a unimodular triangulation that does contain one edge from height $0$ to height $2$ and let us construct another unimodular triangulation $T'$ that contains at least one less such edge. Let $e$ be one of those edges in $T$, which joins $(a,0,0)$ to $(0,b,2q)$ with $a,b\in\{0,q,2q\}$. For $e$ to be primitive we need $q$ to be odd and at last one of $a$ and $b$ to equal $1$. Say $a=1$.
\begin{compactitem}[ $\bullet$ ]
\item If $b=0$ (or $b=2$, which is a symmetric case) then the facet $y=0$ of $2\tetra(p,q)$ must be triangulated joining the edge 
$e$ to the vertices $(0,0,1)$ and $(q,0,1)$ of the fundamental square. The two triangles so formed must in turn be joined to the 
same interior point in the fundamental quadrilateral, namely the point $(p',1,1)$. A flip can be performed on these two tetrahedra 
that changes the edge $e$ to the edge $e'=\conv\{(0,0,1),(q,0,1)\}$.
\item So, only the case $a=b=1$ remains, that is, $e=\conv\{(q,0,0), (0,q,2)\}$, and $q$ odd. Moreover, by what we said in the previous case, we assume $e$ to be the only edge going from height~$0$ to height~$2$; that is, the link of $e$ is a (perhaps non-convex) polygon $P$ contained in the fundamental square. Let $k$ be the number of lattice points in~$P$, which are all in the boundary. We can retriangulate the star of $e$ by triangulating $P$ with $k-2$ triangles and coning to both ends of~$P$. This gives $2k-4$ tetrahedra instead of the $k$ in $T$, and, since $T$ was unimodular, $k$ must be at most four. But, also, $k$ must be even, since each of the halves in this retriangulated star has volume $k/2$. The only possibility is $k=4$, in which case this retriangulation procedure gives a new unimodular triangulation without any edges going from height~$0$ to height~$2$.\vspace{-5.5mm}
\end{compactitem}
\end{proof}

We now focus in the following particular class of empty tetrahedra. 

 \begin{definition}
 \label{defi:tetragonal}
We call the empty tetrahedra of the form $\tetra(\pm 1,q)$ \emph{tetragonal}.
 \end{definition} 

\begin{remark}
\label{rem:tetragonal}
As we show in Corollary~\ref{coro:k=2} below, tetragonal tetrahedra are precisely the ones whose second dilation has a unimodular triangulation. But they can also be characterized 
via their symmetries and via their width:
\begin{compactitem}[ $\circ$ ]
\item  $\tetra(p,q)$ is tetragonal if, and only if, the group of  lattice automorphisms preserving it has order at least eight. (More precisely, the automorphism group of a tetragonal simplex  contains the dihedral group of order eight, acting as the point group of type $\bar{4} 2 m$).
\item $\tetra(p,q)$ if, and only if, it has width~$1$ with respect to at least two of its three pairs of paralell edges.
\end{compactitem}
In both respects the simplices $\tetra(0,1)$ and $\tetra(1,2)$ are special: They are the ones that have width~$1$ with respect to \emph{the three} pairs of opposite edges, and the ones whose symmetry group is the whole tetrahedral group $\bar{4} 3 m$, isomorphic to~$S_4$.%
\end{remark}

\begin{corollary}
\label{coro:k=2}
The following properties are equivalent, for a non-unimodular empty simplex $\tetra(p,q)$:
\begin{compactenum}[ \rm(1) ]
\item $p=\pm 1 \pmod q$. That is,  $\tetra(p,q)$ is tetragonal.
\item All the lattice points in the fundamental square (except of course for two of the vertices of the square) lie in one of the two diagonals.
\item The maximal X-path and the maximal Y-path in the fundamental square are compatible, that is, they intersect properly.
\item $2\tetra(p,q)$ has unimodular triangulations.
\item $k\tetra(p, q)$ has unimodular triangulations for all $k\ge 2$. 
\end{compactenum}
Moreover, if this happens, then unimodular triangulations of\/ $2\tetra(p,q)$ (and of~$k\tetra(p,q)$, for all $k\ge2$)
exist that are standard in the boundary.
\end{corollary}

\begin{proof}
The implications $(1)\Leftrightarrow(2)\Rightarrow(3)$ are obvious: If $p=\pm 1$ the two maximal paths coincide and run along the diagonal of the fundamental square. 
$(5)\Rightarrow(4)\Rightarrow(3)$ are also obvious, the latter in the light of the previous lemma; some unimodular triangulation must contain both maximal paths, hence they are compatible. We now prove $(3)\Rightarrow(1)\Rightarrow(4)$, and postpone the proof of $(3)\Rightarrow(5)$ until Section~\ref{sec:non-standard} (Corollary~\ref{coro:tetragonal}).

For $(3)\Rightarrow(1)$, let $p'=-p\bmod q$ and $p''=-p^{-1}\bmod q$. Without loss of generality assume $p'<q/2$
(that is, $p>q/2$) and, to seek a contradiction, assume $p'>1$. Observe also that $p''\le q-2$ since $p''=q-1$ would imply $p=-{p''}^{-1}\bmod q=1$.

We consider the following four lattice points in the fundamental quadrilateral (see Figure~\ref{fig:compatible}): 
\[\begin{array}{cc}
a=(p',1,1),     & b=(2p',2,1),\\
c=(1,p'',1),& d=(p'+1,p''+1,1).
\end{array}
\]
\begin{figure}
\includegraphics[scale=.5]{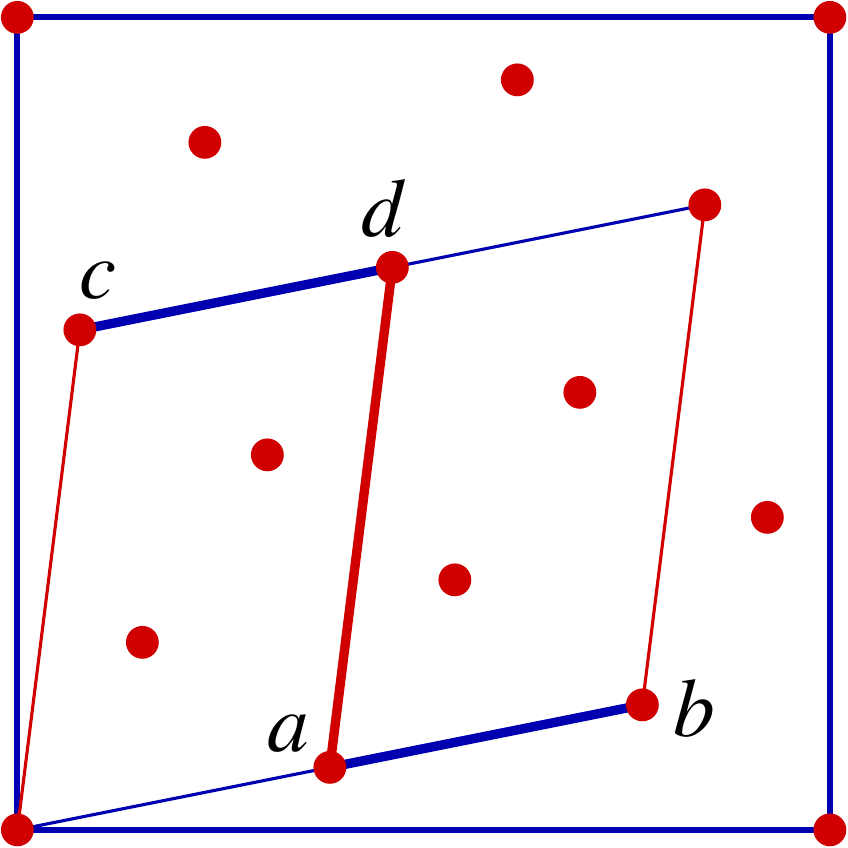}
\caption{If the maximal X and Y paths are compatible in $\tetra(p,q)$ then $p=\pm 1$}
\label{fig:compatible}
\end{figure}
The edges $ab$ and $cd$ are part of the maximal Y-path, while the edge $ad$ is part of the maximal X-path. If the two paths were compatible, the maximal Y-path should go from $b$ to $c$ in a Y-monotone fashion without crossing $ad$, which is impossible.

We finally prove $(1)\Rightarrow(4)$. For this, we consider the following alternative description of the $q+9$ lattice points in $2\tetra(1,q)$, as indicated in Figure~\ref{fig:tetragonal}: 
\begin{compactitem}[ $\bullet$ ]
\item $q+1$ collinear points along one of the diagonals of the fundamental square, starting and finishing with the mid-points of two opposite edges of $2\tetra(1,q)$. We denote these points $a_0,\dots, a_q$.
\item The four vertices of $2\tetra(1,q)$ and the mid-points of the other four edges of $2\tetra(1,q)$. We denote them $b_1,\dots,b_8$, in the circular order in which they appear around the edge $a_0a_q$, and in such a way that $a_0$ is the mid-point of $b_1b_5$ and $a_q$ that of $b_3b_7$.
\end{compactitem}
\begin{figure}
\includegraphics[scale=.5]{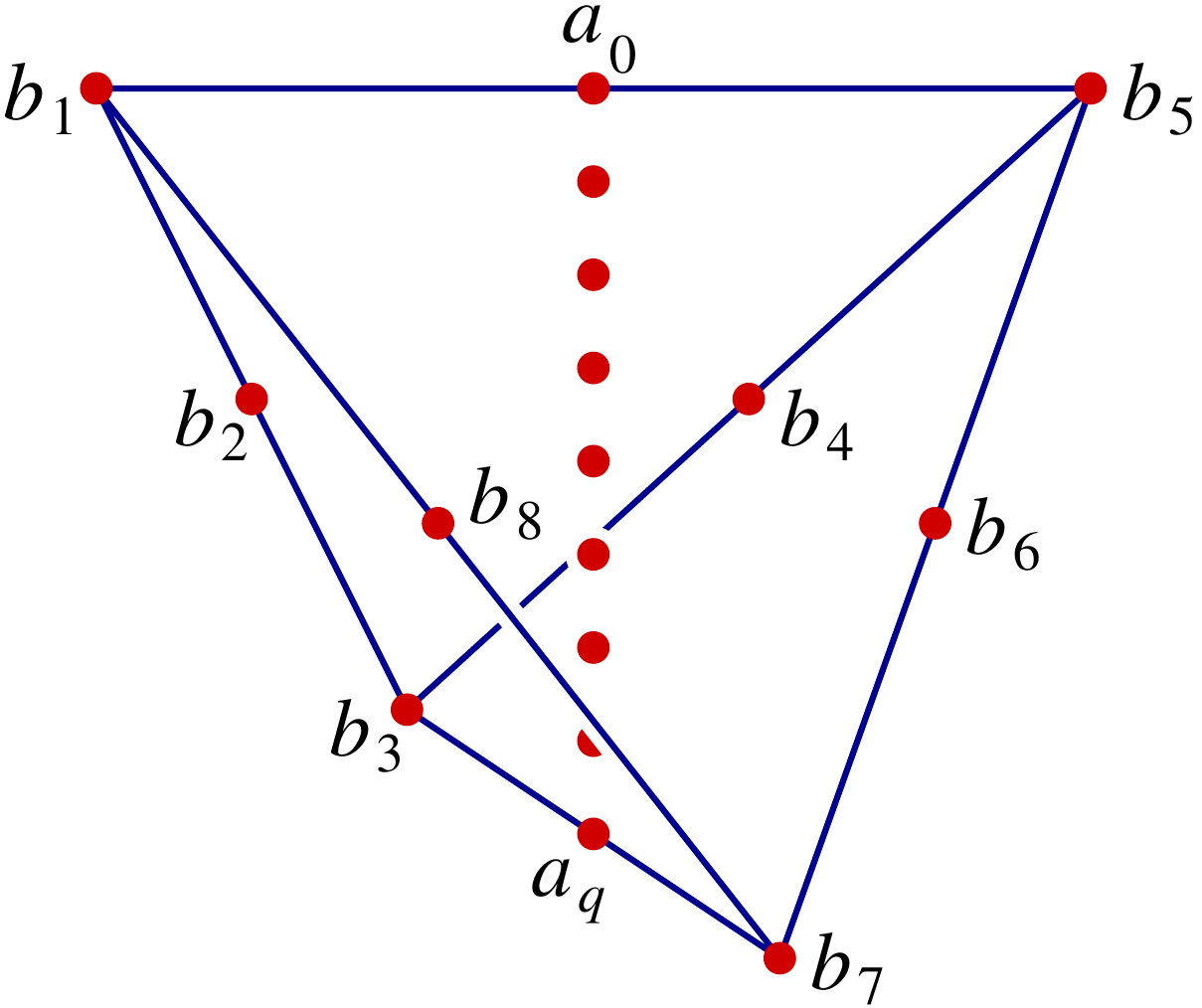}
\caption{The $q+9$ lattice points in $2\tetra(1,q)$}
\label{fig:tetragonal}
\end{figure}
Then an easy way to triangulate  $2\tetra(1,q)$ is to join each of the $q$ segments $a_ia_{i+1}$ with each of the eight segments $b_ib_{i+1}$ (including $b_8b_1$ among the latter). Since the number $8q$ of tetrahedra so obtained equals the normalized volume of $2\tetra(1,q)$, this triangulation, call it $T$, is unimodular.

The only problem that remains is that $T$ is not standard in the boundary, but that can easily be solved. Substitute each pair of tetrahedra form the first column of Table~\ref{table:tetragonal} for the two tetrahedra in the same row of the second column. Each substitution flips between one and the other triangulation of a certain square pyramid, so we get another unimodular triangulation $T'$, which is now standard in the boundary.
\begin{table}
\begin{tabular}{cc|cc}
\multicolumn{2}{c|}{remove  tetrahedra} & \multicolumn{2}{c}{insert  tetrahedra}\\
\hline
$a_0a_1b_2b_3$& $a_0a_1b_3b_4$ & $a_0a_1b_2b_4$& $a_1b_2b_3b_4$ \\
\hline 
$a_{q-1}a_qb_4b_5$& $a_{q-1}a_qb_5b_6$ & $a_{q-1}a_qb_4b_6$& $a_{q-1}b_4b_5b_6$ \\
\hline 
$a_0a_1b_6b_7$& $a_0a_1b_7b_8$ & $a_0a_1b_6b_8$& $a_1b_6b_7b_8$ \\
\hline 
$a_{q-1}a_qb_8b_1$& $a_{q-1}a_qb_1b_2$ & $a_{q-1}a_qb_8b_2$& $a_{q-1}b_8b_1b_2$ \\
\end{tabular}
\vspace{2mm}
\caption{Making the triangulation standard in the boundary}
\label{table:tetragonal}
\end{table}
\end{proof}
 
A triangulation of a (not necessarily lattice) polytope $P$ is called \emph{regular} (or \emph{coherent}, or \emph{convex}) if its simplices are the domains of linearity of a piece-wise linear and convex map $P\to \R$. 
(see, e.g.,~\cite{bgBook}, \cite{deLoeraRambauSantos}, \cite{KKMS3}.) 
Regularity of triangulations is an extremely important property for the applications in toric geometry, which were the original motivation for~\cite{KKMS3}. In particular, it is proved in~\cite{KKMS3} that the triangulations whose existence follows from Theorem~\ref{thm:KMW} can be taken regular. In the constructions of the following sections we do not see an easy way to guarantee that our triangulations are regular, but in the case of tetragonal simplices we do. 

\begin{proposition}
\label{prop:tetragonal_regular}
$2\tetra(1,q)$ has regular unimodular triangulations with standard boundary, for every $q$.
\end{proposition} 

\begin{proof} 
The triangulation that we construct is the same as in the proof of Corollary~\ref{coro:k=2}, except we describe it differently.

We first consider the following subconfiguration containing only $6$ of the $q+9$ lattice points in $2\tetra(1,q)$: the first and last interior points in the maximal path ($a_1$ and $a_{q-1}$, in the notation of the proof of Corollary~\ref{coro:k=2}), plus the four vertices of $2\tetra(1,q)$. This configuration has a unique triangulation that uses both interior points, consisting of the following 8 tetrahedra: the edge $a_1a_{q-1}$ joined to the four edges of $2\tetra(1,q)$ that are not coplanar with it, plus each of $a_1$ and $a_{q-1}$ joined to the two facets $2\tetra(1,q)$ closest to it. This first triangulation is necessarily regular, since $d+3$ points in dimension $d$ never have non-regular triangulations~\cite[Theorem 5.5.1]{deLoeraRambauSantos}. 

Let us call $T_0$ this first triangulation, but considered as a regular \emph{subdivision} of the point configuration $2\tetra(1,q)\cap \Lambda(1,q)$. We now use the idea of \emph{regular refinements}~\cite[Lemma 2.3.6]{deLoeraRambauSantos}: 
If $T_0$ is a regular subdivision of a point configuration and each cell of $T_0$ is refined using a weight
vector  $\omega$ (the same in all cells) then the refinements of the different cells agree on the intersections and the subdivision obtained is regular.
 
In our case, the perturbing weights $\omega$ are easy to obtain: For the cells containing $a_1a_{q-1}$ any choice that creates a \emph{full} triangulation (one using all the intermediate points as vertices) does the job. For the other four tetrahedra, any choice that creates the standard triangulation in the boundary is valid. For example, taking as $\omega$ the Delaunay weights works in both cases.
 \end{proof}

\section{Unimodular triangulations of $k\tetra(p,q)$ exist for all $k\ge 4$}%
\label{sec:non-standard}

Here we show that for every empty tetrahedron $\tetra(p,q)$ and every $k\ge 4$ there are regular triangulations of $k\tetra(p,q)$. However, the triangulations constructed in this section are not guaranteed to be standard in the boundary. Hence they do not (automatically) provide unimodular triangulations of $kP$ for a lattice polytope $P$.

First observe that the set of lattice points in $k\tetra(p,q)$ has a structure very similar to those in  $2\tetra(p,q)$:
\begin{compactitem}[ $\bullet$]
\item It contains ${k+3 \choose 3}$ points of the lattice $q\Z\times q\Z\times \Z$, namely the points $(aq,bq,c)$ for $c=0,\dots,k$, $a=0,\dots,k-c$, $b=0,\dots,c$.
\item These points define ${k+2 \choose 3}$ translates of the fundamental square, namely 
\[
(aq,bq,c)+\conv\{(0,0,0), ((q,0,0), (0,q,0), (q,q,0)\}, \\
\]
for all $ c=1,\dots,k-1$, $a=0,\dots,k-c-1$, $b=0,\dots,c-1$.
In each of these fundamental squares we have $q-1$ interior lattice points, translates of those in the initial fundamental square.
\end{compactitem}
Our technique to triangulate $k\tetra(p,q)$ generalizes what we did in the previous section. Thus we first decompose $k\tetra(p,q)$ into $k$ horizontal layers, where the $c$-th layer equals
\[
\tetra(p,q)\cap \{(x,y,z)\in \R^3:c\le z\le c+1\}.
\]
See Figure~\ref{fig:layers} for an illustration.
\begin{figure}
\includegraphics[scale=.5]{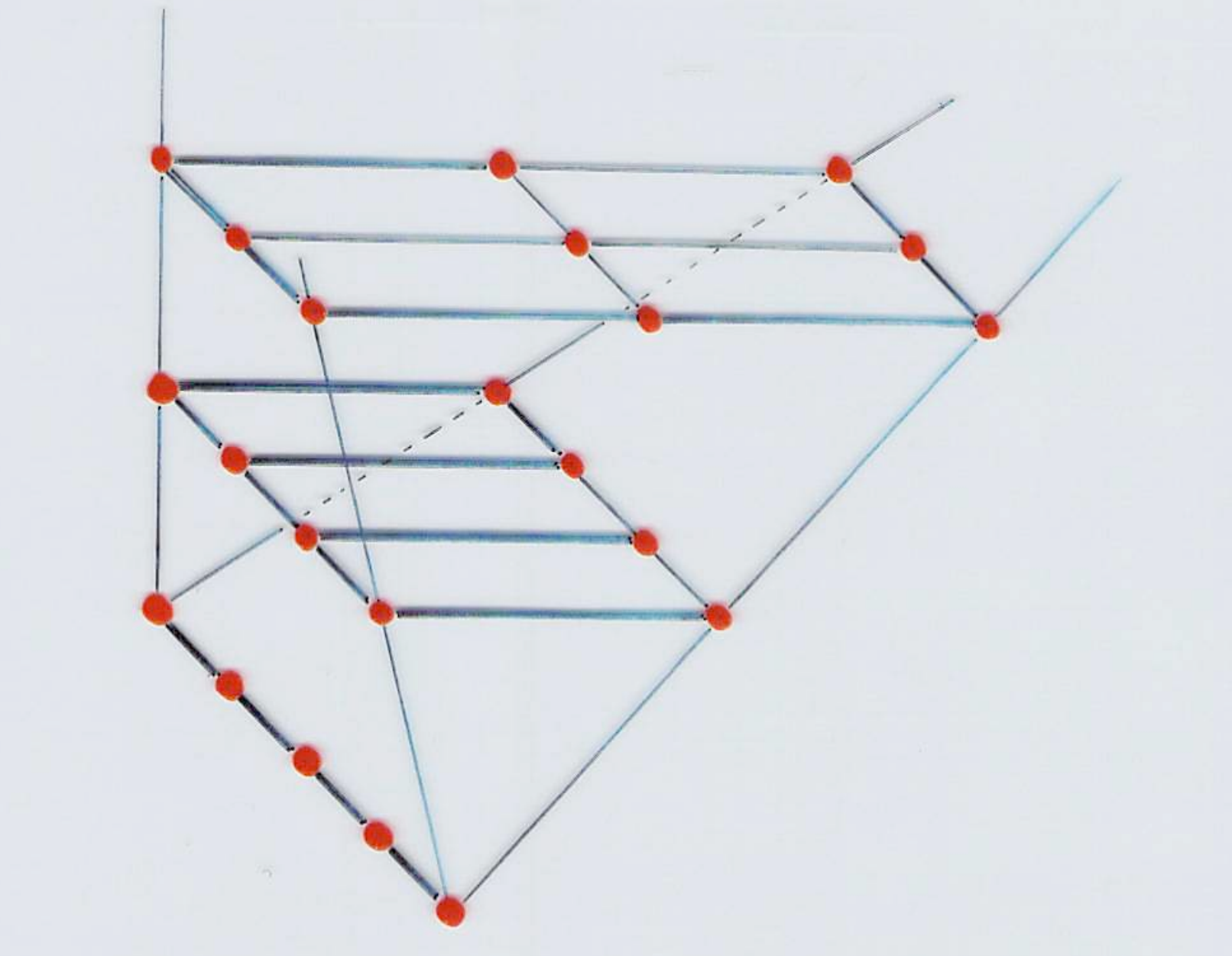}
\caption{Decomposition of $4\tetra(p,q)$ into layers.}
\label{fig:layers}
\end{figure}
The $c$-th layer ($c=2,\dots,k-1$) can then be decomposed into either $2c-1$ triangular prisms in the X direction, or into $2(k-c)-1$ prisms in the Y direction; we call these  \emph{toblerone prisms}.  
The first and last layers are themselves a toblerone prism each, in the X and the Y direction, respectively (see Figure~\ref{fig:toblerone}).
\begin{figure}
\includegraphics[scale=.6]{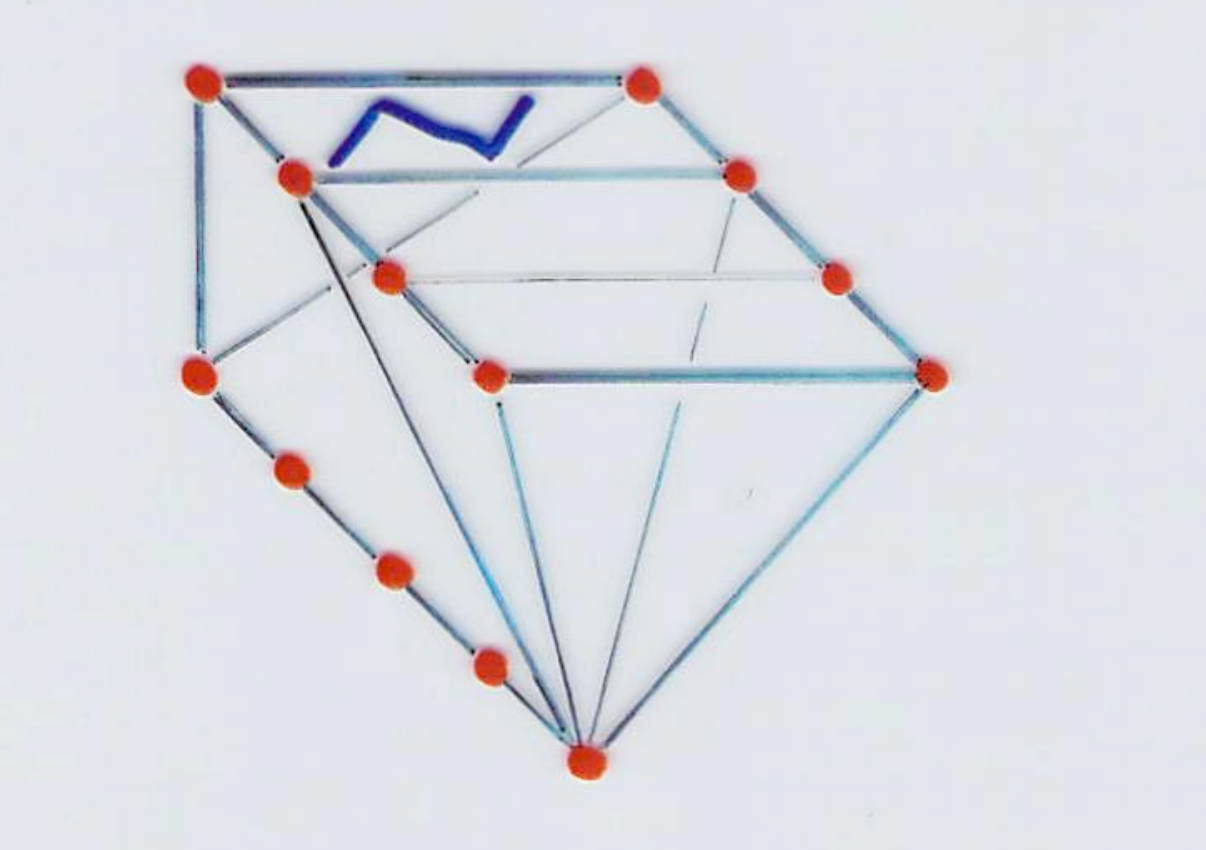}
\caption{The bottom toblerone prism in $4\tetra(p,q)$.}
\label{fig:toblerone}
\end{figure}

The following lemma describes all the possible triangulations of a toblerone prism, and generalizes the construction of Corollary~\ref{coro:k=2}:

\begin{lemma}
\label{lemma:toblerone}
Let $\Delta$ be a toblerone prism in the dilated simplex $k\tetra(p,q)$. That is, let $\Delta$ be (lattice congruent to) the convex hull of the following six points, for certain natural numbers $c,d\in \N$:
\[
(0,0,1),(cq,0,1),(0,q,1),(cq,q,1), (0,0,0),(dq,0,0).
\]
(In $k\tetra(p,q)$ we have $d=c\pm1$, but that will not be used for this lemma.)%

\noindent
Suppose that a triangulation $T_0$ of the top rectangle 
\[
\conv\{(0,0,1),(cq,0,1),(0,q,1),(cq,q,1)\}
\]
of $\Delta$ is given. Then every triangulation $T$ of $\Delta$ that extends $T_0$ has the following description:
\begin{compactitem}[ $\bullet $]
\item For each $i=0,\dots,d-1$, there is a  Y-monotone path in $T_0$ going from a vertex of the form $(x_i,0,1)$ to one of the form $(x'_i,q,1)$ joined to the edge $\conv\{(qi,0,0), (q(i+1),0,0)\}$ and these $d$ paths are each above the previous one in the X direction (in the weak sense; parts of one path can coincide with the next one. In fact, all the paths could well be the same~one).%
\item The point $(qi,0,0)$ is joined to the part of $T_0$ lying between the $(i-1)$-th and the $i$-th paths.
\end{compactitem}
Conversely, for any choice of $d$ paths in $T_0$ as above there is a unique triangulation $T$ of $\Delta$ obtained from them.

Moreover, the triangulation $T$ so obtained is unimodular if, and only if, all the Y-monotone paths used are maximal (that is, they have length $q$).
\end{lemma}

\begin{proof}
If $T$ is a triangulation that extends $T_0$ then the link of each edge $\conv\{(qi,0,0),(q(i+1),0,0)\}$ is a Y-monotone path crossing the top rectangle from one side to the other, and these $d$-paths are necessarily each above the previous one in the X-direction. Also, the point $(qi,0,0)$ must be joined to all the triangles in $T_0$ lying between the  $(i-1)$-th and the $i$-th paths, if any. Conversely, any choice of such paths clearly produces a triangulation.

For the unimodularity, observe that all the tetrahedra obtained joining a triangle in the top rectangle to a point in the bottom edge are automatically unimodular. On the other hand, a tetrahedron obtained joining  $\conv\{(qi,0,0),$ $(q(i+1),0,0)\}$ to two points
 $(a,a',1)$ and $(b,b',1)$ has volume $b'-a'$. Hence, the star of $\conv\{(qi,0,0),(q(i+1),0,0)\}$ is unimodular if, and only if, the corresponding Y-monotone path is maximal.
\end{proof}
 
In the following statement we call a \emph{dissection} of a polytope any covering by a union of simplices that are contained in it and whose relative interiors are disjoint. That is, a dissection is ``almost'' a triangulation, except that facets (and lower dimensional faces) of different simplices may intersect improperly.

\begin{corollary}
\label{coro:dissections}
For every lattice $3$-polytope and every dilation factor $k\ge 2$, $kP$ has a unimodular dissection.
\end{corollary}

\begin{proof}
Consider first a triangulation $T$ of $P$ into empty tetrahedra. Then, partition each dilated simplex $k\Delta$, $\Delta\in T$, into toblerone prisms as above, and triangulate each prism independently into unimodular tetrahedra.
\end{proof}

In order to achieve true triangulations this way we need to ensure that the triangulations of the individual prisms agree on their boundaries. For this we observe that, with the notation of Lemma~\ref{lemma:toblerone}, the triangulation obtained in the two non-horizontal quadrilaterals of each toblerone prism are determined by the following property: Each edge $\conv\{(qi,0,0),(q(i+1),0,0)\}$ is joined to the points $(x_i,0,1)$ and $(x'_i,q,1)$ at which the $i$-th Y-path starts and ends. With this in mind we consider the following two cases, in which we have some control on the boundary triangulation:

\begin{lemma}
\label{lemma:arbitrary_boundary}
In the conditions of Lemma~\ref{lemma:toblerone} suppose that $T_0$ contains a maximal Y-path joining every point of the form $(x,0,1)$ to every point of the form $(x',q,1)$. Then a unimodular triangulation $T$ of the toblerone can be obtained extending any arbitrarily chosen triangulation of the non-horizontal boundary of the toblerone.
\end{lemma}

One way to obtain a $T_0$ in the conditions of the statement is to choose a single maximal Y-path to which all the vertices of the fundamental squares are joined. Figure~\ref{fig:arbitrary_boundary} is an illustration of this.
\begin{figure}
\includegraphics[scale=.5]{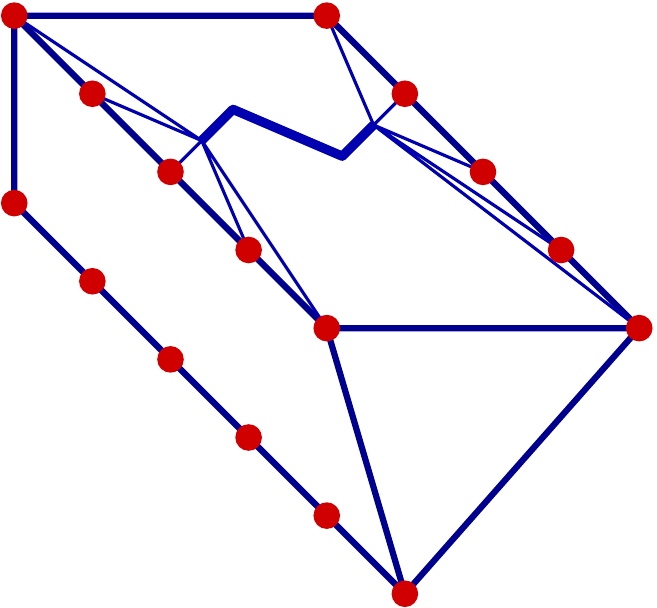}
\caption{Choosing a $T_0$ that contains maximal Y-paths joining every point $(x,0,1)$ to every point $(x',q,1)$ gives complete freedom for the triangulation of the sides of the toblerone.}
\label{fig:arbitrary_boundary}
\end{figure}

\begin{proof}
Here we have complete freedom to choose the points $(x_i,0,1)$ and $(x'_i,q,1)$ for each Y-path.
\end{proof}

\begin{lemma}
\label{lemma:alternating}
In the conditions of Lemma~\ref{lemma:toblerone} suppose that $d=c+1$, so that it makes sense to say whether $T$ has standard boundary. Suppose that $T_0$ triangulates the fundamental squares alternating between using the maximal X-path and the maximal Y-path from one to the next. Then, a unimodular triangulation $T$ of the toblerone can be obtained having standard boundary except for four boundary edges, one on either extreme of either side of the toblerone.
\end{lemma}

\begin{proof}
The boundary is standard if we can choose  $x_i=x'_i=qi$ as the extremes for the $i$-th maximal Y-path. In the conditions of the statement that is possible except perhaps for $i=0$ and $i=c$. In fact, we can get $x_0=x'_0=0$ if and only if the first fundamental square uses its maximal Y-path and we can have $x_c=x'_c=qc$ if and only if the last fundamental square uses its maximal Y-path.
\end{proof}

\begin{figure}
\includegraphics[scale=.5]{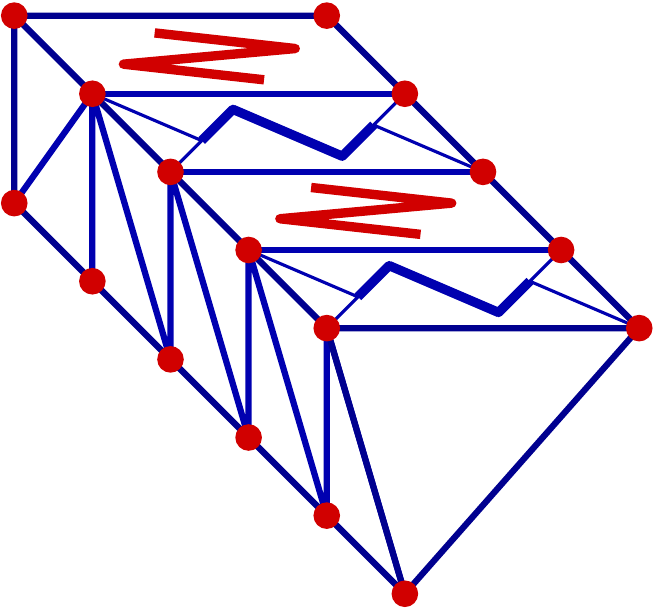}
\caption{Choosing a $T_0$ that alternates between maximal Y-paths and maximal X-paths in consecutive fundamental squares allows to get an almost standard boundary.}
\label{fig:alternating_toblerone}
\end{figure}

\begin{corollary}
\label{coro:non-standard}
$k\tetra(p,q)$ has a unimodular triangulation, for all $p,q$ and for all $k\ge 4$. Moreover, this triangulation can be made to have no more than $4$ non-standard edges in the boundary.
\end{corollary}

\begin{proof}
Choose a height $c$ between $2$ and $k-2$. Let $\tetra(p,q)_{[0,c]}:=\tetra(p,q)\cap \{(x,y,z)\in \R^3: z\le c\}$ and $\tetra(p,q)_{[c,k]}:=\tetra(p,q)\cap \{(x,y,z)\in \R^3: z\ge c\}$ the lower and upper parts. Divide the lower part into toblerones in the X direction and the upper into toblerones in the Y direction. Triangulate all the fundamental squares at height less than $c$ using their maximal Y-path, and those at height greater than $c$ using the maximal X-path. Triangulate the fundamental squares alternating the 
the maximal X- and Y-paths, as shown in Figure~\ref{fig:alternating}.

\begin{figure}
\includegraphics[scale=.5]{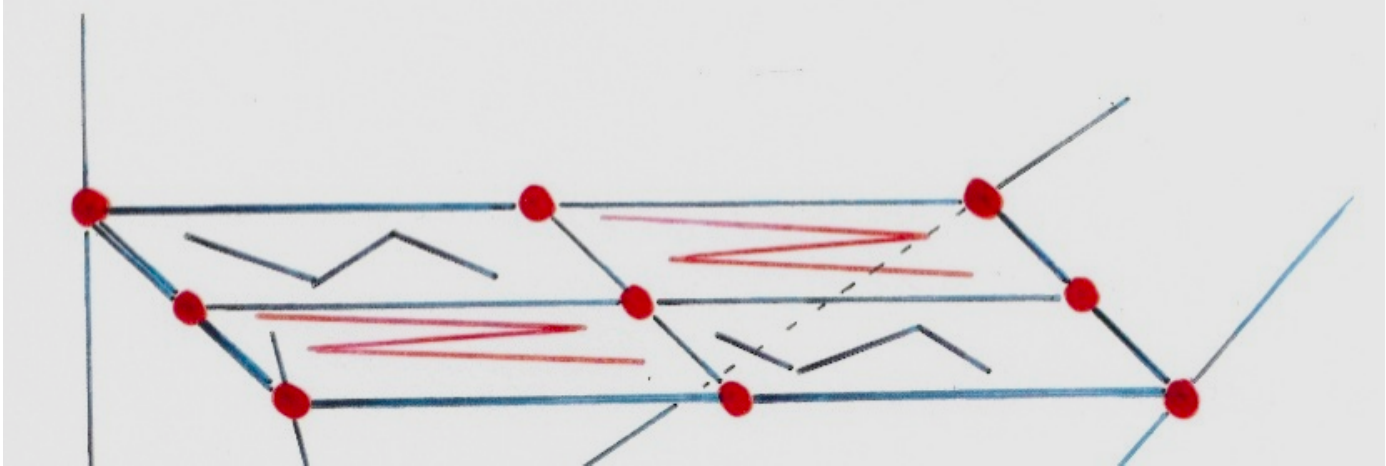}
\caption{Alternating X-chains and Y-chains.}
\label{fig:alternating}
\end{figure}

We now triangulate the toblerones in a way compatible with the triangulations of the fundamental squares, and starting with those incident to the fundamental squares at height $c$. For these toblerones, the triangulations of the fundamental squares are in the conditions of Lemma~\ref{lemma:alternating}. In particular, we triangulate them with at most four non-standard edges in the boundary of each, eight in total in the boundary of $k\tetra(p,q)$. However, this eight can be turned to a four if we take into account that each of the four corner fundamental squares at height $c$ produces a non-standard edge only in one of the two layers incident to it, not both.

Once this is done, since the rest of the toblerones are in the conditions of Lemma~\ref{lemma:arbitrary_boundary}, we can triangulate them so that their boundary triangulations is on the one hand standard in the boundary of $k\tetra(p,q)$ and on the other hand compatible with the toblerones adjacent to them if they were triangulated earlier.
See Figure~\ref{fig:non-standard} for a picture illustrating the whole process.
\end{proof}

\begin{figure}
\includegraphics[scale=.50]{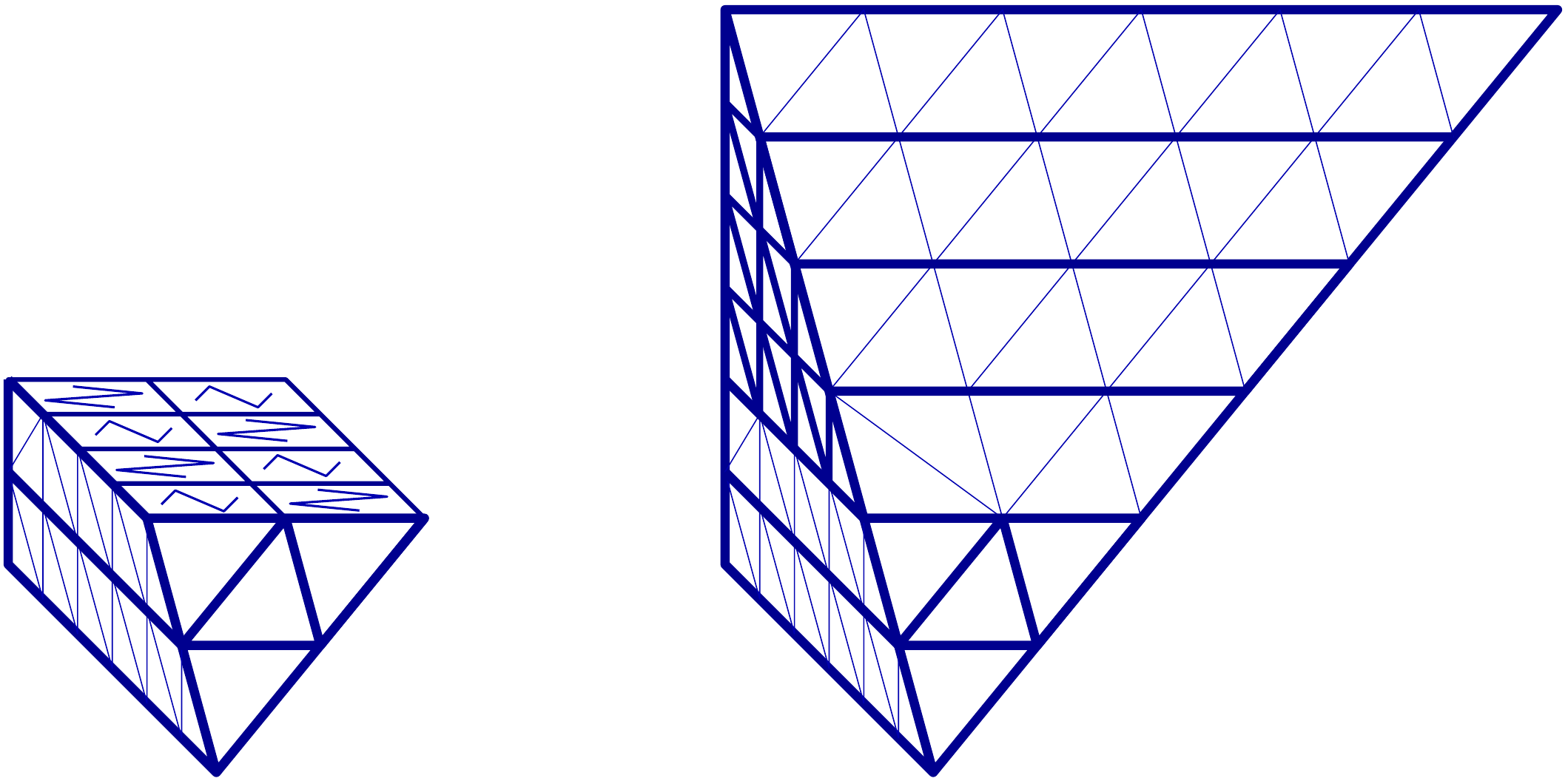}
\caption{Illustration of the proof of Corollary~\ref{coro:non-standard}, for $k=6$ and $c=2$. The bottom part (left) is divided into prisms in the X direction and the top part into prisms in the Y direction.}
\label{fig:non-standard}
\end{figure}

One very special case of the procedure of Lemma~\ref{lemma:alternating} and Corollary~\ref{coro:non-standard} arises for tetragonal empty simplices $\tetra(1,q)$. Since in a tetragonal simplex the maximal X-path and the maximal Y-path coincide, the construction of Lemma~\ref{lemma:alternating} gives the following:
 
\begin{corollary}
\label{coro:tetragonal}
$k\tetra(1,q)$ has a unimodular triangulation with standard boundary, for all $q$ and for all $k\ge 2$.
\qed
\end{corollary}

\section{Unimodular triangulations of $k\tetra(p,q)$ with standard boundary exist for $k\in \N\setminus\{1,2,3,5,7,11\}$}
\label{sec:standard}

In this section we combine two ideas which together will lead to the following.

\begin{theorem}
\label{thm:standard}
Let $\tetra(p,q)$ be any empty simplex. Then for every $k\in \N\setminus\{1,2,3,5,7,11\}$, \ $k\tetra(p,q)$ has a 
unimodular triangulation with standard boundary.
\end{theorem}

We will show this be dealing first with all the composite values of $k$ and then with all the values that can be written as a sum of two smaller ones.

\subsection{Composite $k$}
Kantor and Sarkaria~\cite{KantorSarkaria} have shown that $4\tetra(p,q)$ has a unimodular triangulation with standard boundary. 
In this section we rework their ideas and show that they are sufficient to establish the result for any for any composite $k$
in place of~$4$.

Consider the fundamental square corresponding to an empty simplex $\tetra(p,q)$. We recall the following definitions from~\cite{KantorSarkaria}. As in previous sections, we let $p':=-p \bmod q$ and $p'':=-p^{-1} \bmod q$. We also introduce $p''':=p^{-1} \bmod q = -p'' \bmod q$.

\begin{definition}
\label{defi:lati_long}
We call \emph{latitudes} the lattice lines of direction $(p',1,0)$ if $p'<q/2$ or those of direction $(-p,1,0)$ otherwise.
We call \emph{longitudes} the lattice lines of direction $(1,p'',0)$ if $p''<q/2$ or $(1,-p''',0)$ otherwise.
\end{definition}

\begin{figure}
\includegraphics[scale=.5]{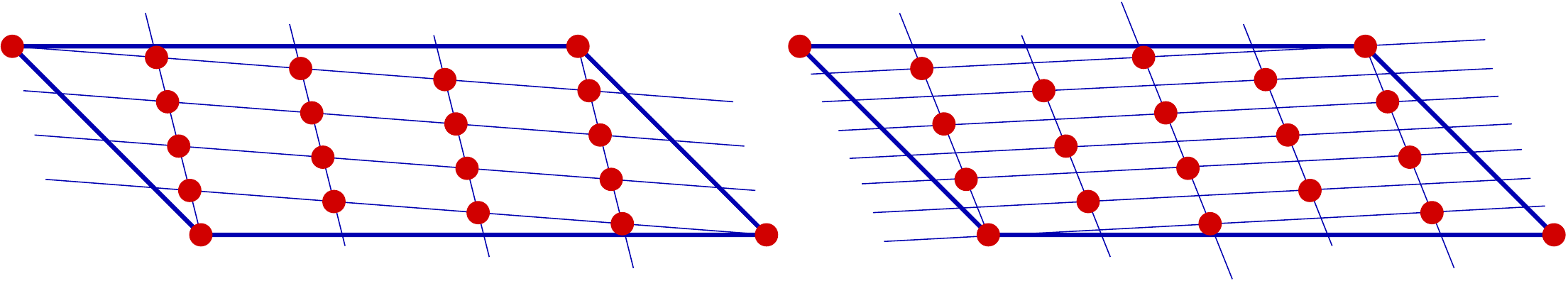}
\caption{Latitudes and longitudes for $\tetra(4,17)$ (left) and for $\tetra(5,17)$ (right).}
\label{fig:lati_long}
\end{figure}

See Figure~\ref{fig:lati_long} for some examples. An X-monotone (resp., Y-monotone) path is called \emph{quasi-maximal} if consecutive points in the path are either consecutive in the X-order (resp., Y-order) of lattice points in the fundamental square or they lie in consecutive longitudes (resp., latitudes). 

The significance of this concept is the following. Recall that a tetragonal empty tetrahedron is one equivalent to $\tetra(1,q)$.

\begin{lemma}
\label{lemma:quasi-maximal}
In the conditions of Lemma~\ref{lemma:toblerone}, if the Y-monotone paths chosen as links of the edges $\conv\{(qi,0,0),(q(i+1),0,0)\}$ are quasi-maximal, then all the tetrahedra in the triangulation $T$ are tetragonal.
\end{lemma}

\begin{proof}
As in the proof of Lemma~\ref{lemma:toblerone}, the tetrahedra obtained joining a triangle in a horizontal plane to a vertex in the 
next one are automatically unimodular. For the others, the condition that the Y-path is almost maximal clearly implies that every tetrahedron obtained, if dilated, has all lattice points in the diagonals of its own fundamental squares. This is one of the characterizations of tetragonal empty tetrahedra, as seen in Corollary~\ref{coro:k=2}.
\end{proof}

We now give without proof  
a fundamental result of Kantor and Sarkaria.
\begin{lemma}[\cite{KantorSarkaria}]
\label{lemma:compatible}
For every $p$ and $q$, the fundamental square has a triangulation that contains both a quasi-maximal X-path and a quasi-maximal Y-path.
\end{lemma}

\begin{corollary}
For every empty simplex $\tetra(p,q)$ and every dilation factor $k\in \N$, $k\tetra(p,q)$ admits a triangulation into tetragonal tetrahedra and with standard boundary.
\end{corollary}

\begin{proof}
Triangulate as in Corollary~\ref{coro:non-standard} except in each fundamental square we choose the triangulation of the previous lemma and choosing all the X-paths and Y-path quasi-maximal. Then by Lemma~\ref{lemma:quasi-maximal} the triangulation obtained has only tetragonal tetrahedra and standard boundary.
\end{proof}

\begin{corollary}
For every empty simplex $\tetra(p,q)$ and every dilation factor $k\in \N$ that is composite, $k\tetra(p,q)$ admits a unimodular triangulation with standard boundary.
\end{corollary}

\begin{proof}
Let $k=k_1k_2$ be a factorization of $k$. By the previous Corollary, $k_1\tetra(p,q)$ has a triangulation into tetragonal empty tetrahedra. By a second dilation of factor $k_2$, and applying Corollary~\ref{coro:k=2} to each of the tetragonal tetrahedra obtained, we get a unimodular triangulation of $k\tetra(p,q)$ with standard boundary,
\end{proof}

\subsection{$k=k_1 + k_2$}

\begin{theorem}
\label{thm:k1+k2}
Let $\tetra(p,q)$ be any empty simplex, and let $k=k_1+k_2$. If $k_1\tetra(p,q)$ and $k_2\tetra(p,q)$ both have a 
unimodular triangulation with standard boundary then so does $k\tetra(p,q)$.
\end{theorem}

\begin{proof}
We cut from $k\tetra(p,q)$ a copy of $k_1\tetra(p,q)$ incident to the vertex $(0,0,0)$ and a copy of $k_2\tetra(p,q)$ incident to the vertex $(kq,0,0)$. Those parts, by assumption, admit unimodular triangulations with standard boundary. For the rest, we partition into horizontal strips and then into toblerone simplices, as was done in Section~\ref{sec:non-standard}, except now all the toblerones go in the Y-direction. If we triangulate all fundamental squares using their maximal X-path, we get unimodular triangulations of all the toblerones, all with standard boundary (it has to be noted that the bottom-most toblerone is actually a square pyramid, but that does not interfere with the triangulation process). See Figure~\ref{fig:k1+k2}.
\begin{figure}
\includegraphics[scale=.5]{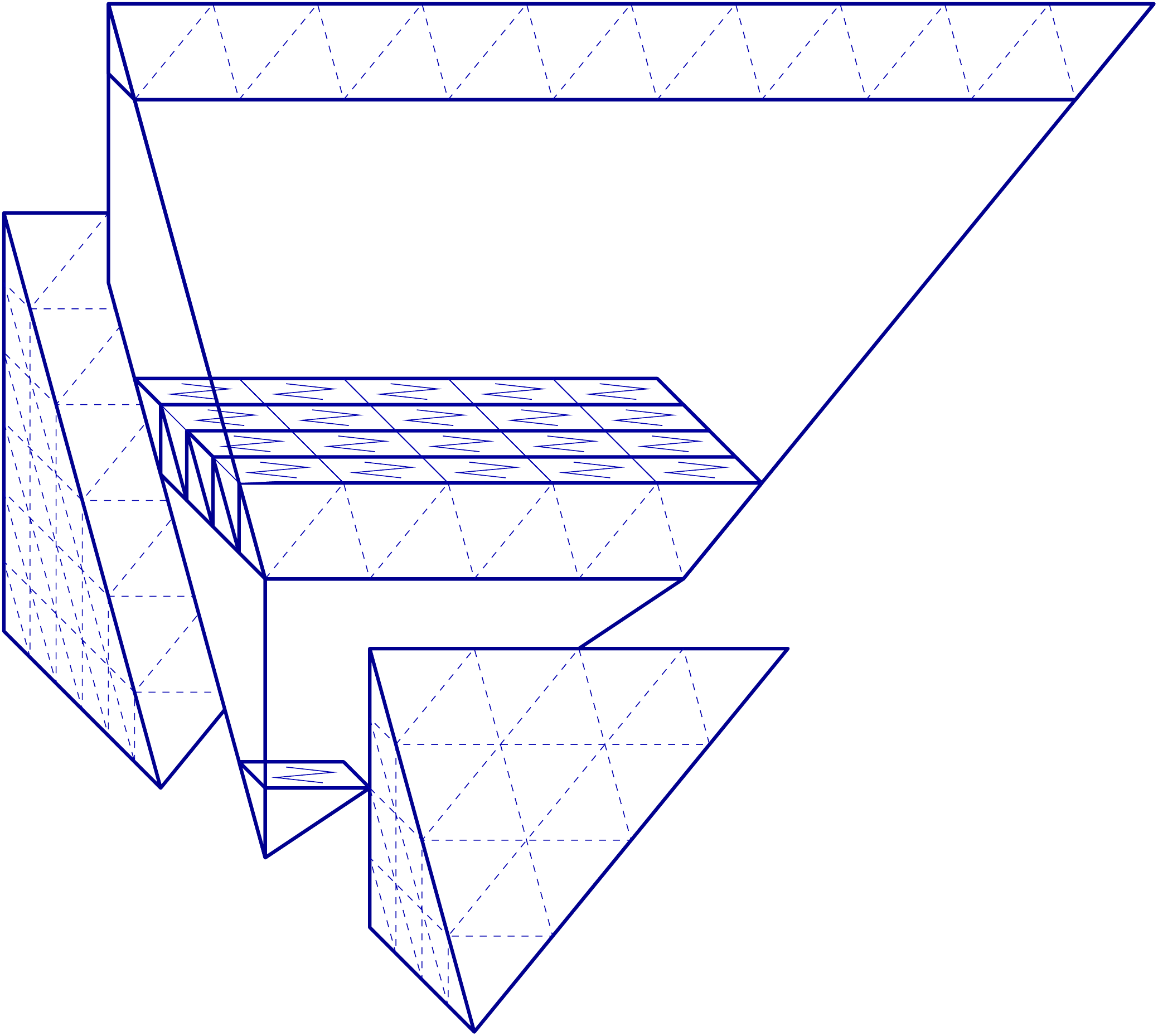}
\caption{Triangulating $(k_1+k_2)\tetra(p,q)$.}
\label{fig:k1+k2}
\end{figure}
\end{proof}

\section{Unimodular triangulations of $kP$ exist for every lattice $3$-polytope $P$ and every $k\in \N\setminus\{1,2,3,5\}$}
\label{sec:semi-standard}

In Section~\ref{sec:standard} we insisted our triangulations of $k\tetra(p,q)$ to have standard boundary because that guarantees that after dilating a triangulation $T$ of a lattice polytope $P$ we can refine each dilated simplex $k\Delta$, with $\Delta\in T$ independently and still get a triangulation of $kP$. In this section we extend this idea as follows: Instead of using the standard triangulation of every dilated triangle, we can use any triangulation that has the full (lattice) symmetries of the triangle, as long as we use the same one in every triangle.
The triangulation that we choose is the following one.

\begin{definition}
\label{quasi-standard}
Let $k\ge 7$ and let $\Gamma$ be a unimodular triangle. We call \emph{quasi-standard} triangulation of $k\Gamma$ the one obtained from the standard one by flipping the six edges whose barycentric coordinates are permutations of $\{\frac{1}{2}, \frac{5}{2}, k-3\}$. See Figure~\ref{fig:quasi-standard}.
\end{definition}
\begin{figure}
\includegraphics[scale=.5]{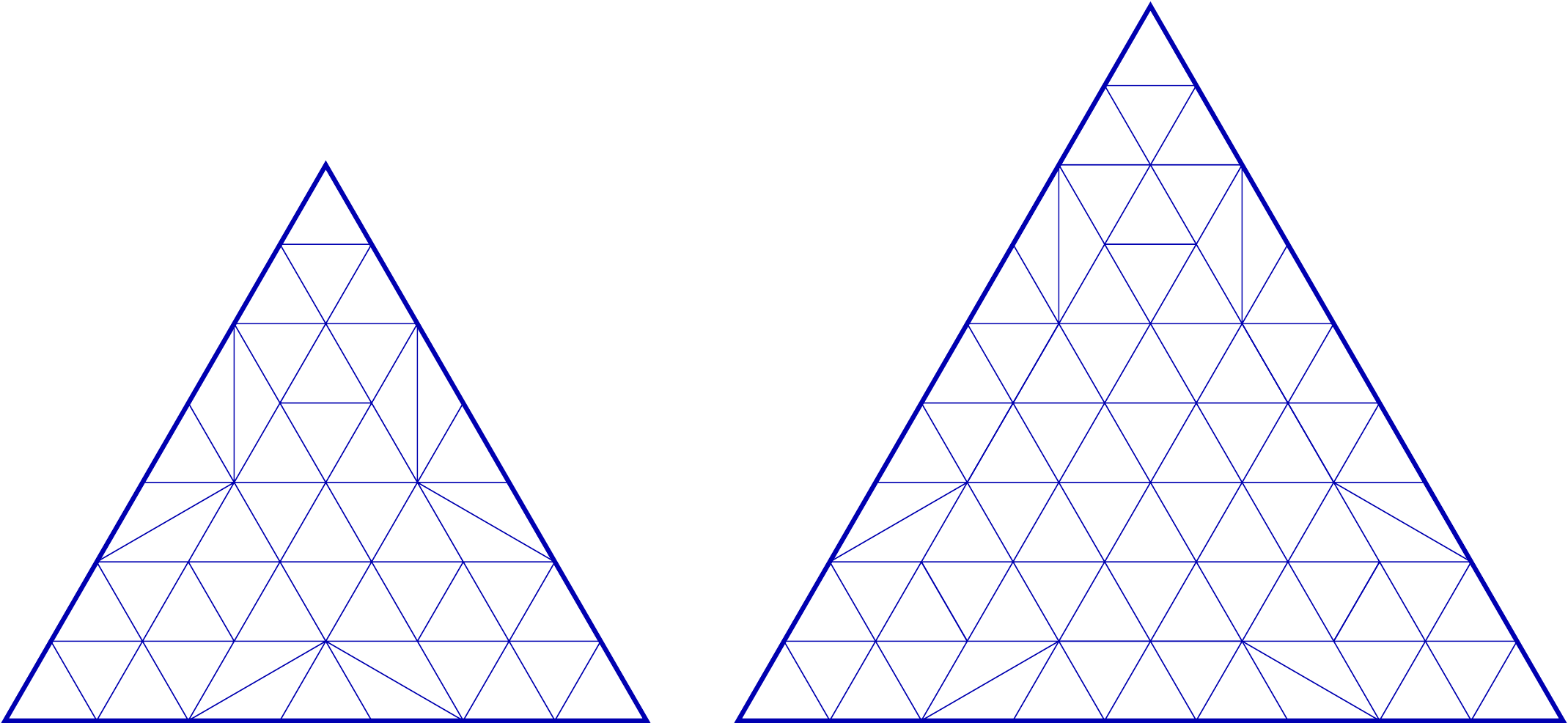}
\caption{Quasi-standard triangulation of a dilated triangle, for $k=7$ and $k=9$.}
\label{fig:quasi-standard}
\end{figure}

\begin{theorem}
\label{thm:quasi-standard}
For every $k\ge 7$ and for every empty tetrahedron $\tetra(p,q)$, $k\tetra(p,q)$ has a unimodular triangulation with quasi-standard boundary.
\end{theorem}

\begin{proof}
We triangulate $k\tetra(p,q)$ with the same technique of Corollary~\ref{coro:non-standard}, taking $c=3$: that is, 
we first divide $k\tetra(p,q)$ into $k$ horizontal layers, then divide the bottom three layers into toblerone prisms in the X direction and the upper $k-3$ into toblerone prisms in the Y direction. In the interface between the two parts we triangulate the $3\times (k-3)$ fundamental squares as follows: we use the maximal Y-path in the four corner plus the $k-5$ interior squares, and we use the maximal X-path in the other $2k-8$ squares (see Figure~\ref{fig:alternating2}). 
\begin{figure}
\includegraphics[scale=.8]{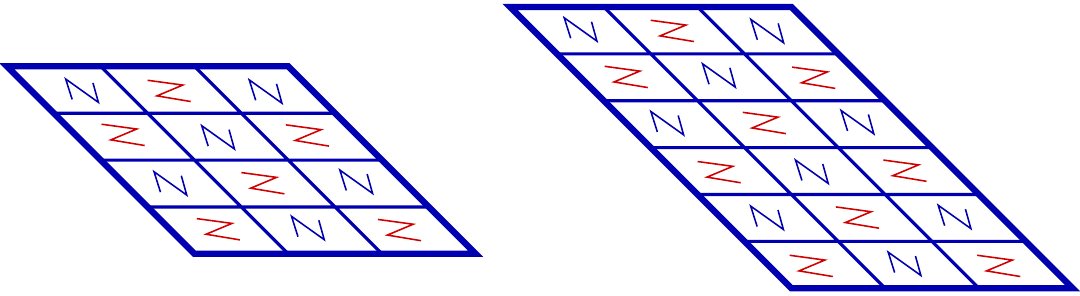}
\caption{The triangulation of the interface between the upper and lower parts of $k\tetra(p,q)$, in the proof of Theorem~\ref{quasi-standard}.}
\label{fig:alternating2}
\end{figure}
Since every row and every column contain squares both using the maximal X-path and the maximal Y-path, all the toblerone prisms incident to this interface can be triangulated unimodularly. Moreover, these unimodular triangulations can be taken so that in the boundary of  $k\tetra(p,q)$ only four non-standard edges arise, but these are precisely the four that we want in order to get a quasi-standard triangulation. In the rest of horizontal planes in between prisms we use the triangulations of Lemma~\ref{lemma:arbitrary_boundary}. Thus we get complete freedom as to how the rest of the boundaries of tob\-lerone prisms get triangulated. We of course choose to triangulate them in accordance to our choice of quasi-standard triangulation, arriving to the situation of Figure~\ref{fig:intermediate}:
\begin{figure}
\includegraphics[scale=.5]{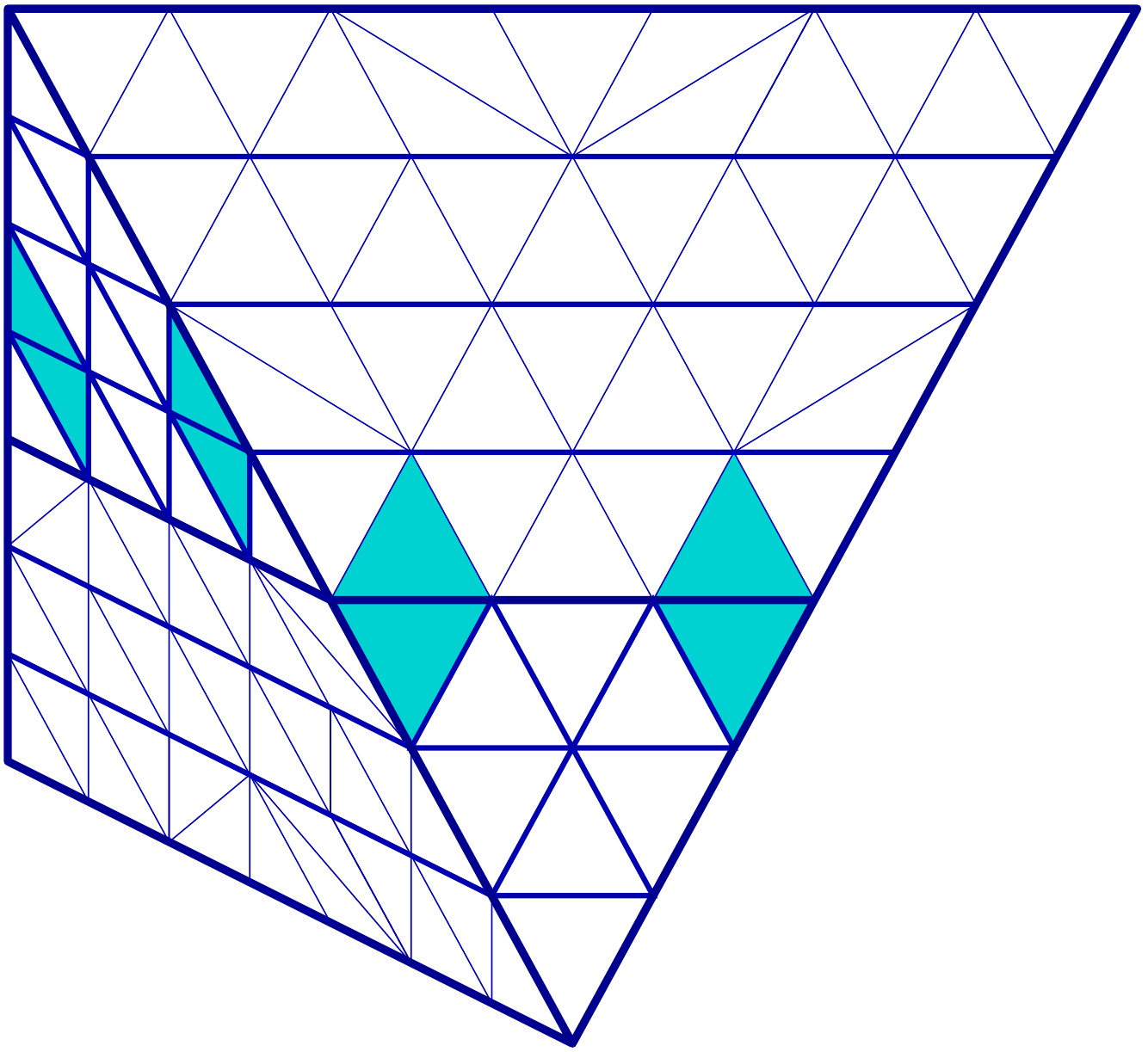}
\caption{The intermediate boundary triangulation in the proof of Theorem~\ref{quasi-standard}.}
\label{fig:intermediate}
\end{figure}
We have a unimodular triangulation of $k\tetra(p,q)$ in which each boundary face has a triangulation that uses four of the six non-standard edges that we want to use. The eight missing ones (two on each facet) are horizontal edges in the triangular faces of two toblerone prisms each. See again Figure~\ref{fig:intermediate}, where the triangles incident to those edges are shaded (only four are seen in the figure; there are another four on the back).

But, as it turns out, those eight edges can be ``flipped'' in our triangulation. Indeed, the two boundary triangles incident to each of them are joined to the same point, namely, the unique point at lattice distance one from those triangles in the corresponding toblerones. By flipping these eight edges we get a unimodular triangulation with quasi-standard boundary.
\end{proof}

\bibliographystyle{plain}

\begin{thebibliography}{99}
 

\bibitem{bgBook}
Winfried Bruns and Joseph Gubeladze:
\emph{Polytopes, Rings, and K-Theory},
Monographs in Mathematics, Springer, 2009. 

\bibitem{deLoeraRambauSantos}
Jes\'us~A. De~Loera, J\"org Rambau, and Francisco Santos:
\emph{Triangulations: Structures for Algorithms and Applications}.
Springer, 2010. 

\bibitem{KantorSarkaria}
Jean-Michel Kantor and Karanbir S.~Sarkaria:
On primitive subdivisions of an ele\-mentary tetrahedron,
\emph{Pacific J. Math.} 211 (2003), 123--155. 

\bibitem{KKMS3}
Finn~F.~Knudsen:
Construction of nice polyhedral subdivisions, 
Chapter 3 of “Toroidal {Embeddings}~{I}” by G.~R.\ Kempf, F.~F.\ Knudsen, D.~Mumford, and B.~Saint-Donat, 
\emph{Lecture Notes in Mathematics} 339 (1973), 109--164.
 

\bibitem{Reeve}
John E. Reeve:
On the volume of lattice polyhedra, 
\emph{Proc. London Math. Soc.} 7 (1957), 378--395.

\bibitem{Reznick}
Bruce Reznick:
Clean lattice tetrahedra, 
preprint, June 2006, 21~pages, \url{http://arxiv.org/abs/math/0606227}.

\bibitem{Scarf}
Herbert E. Scarf:
Integral polyhedra in three space,
\emph{Math.\ Oper.\ Res.}, 10 (1985), 403--438.

\bibitem{Sebo}
Andr\'as Seb\H{o}:
An introduction to empty simplices, 
in: Proceedings of IPCO 7, 
\emph{Lecture Notes in Computer Science} 1610 (1999), 400--414.

\bibitem{White}
George~K.~White:
Lattice tetrahedra,
\emph{Canadian J.\ Math.} 16 (1964), 389--396.
 

\end{thebibliography}

\end{document}